\documentclass[psamsfonts]{amsart}

\usepackage{amssymb,amsfonts,amscd}

\usepackage[colorlinks,linktocpage]{hyperref}
\hypersetup{linkcolor=[rgb]{0,0,0.715}}
\hypersetup{citecolor=[rgb]{0,0.715,0}}

\newtheorem{thm}{Theorem}[section]

\newtheorem{prop}[thm]{Proposition}
\newtheorem{lem}[thm]{Lemma}

\newtheorem*{thm1}{Theorem A}
\newtheorem*{thm2}{Theorem B}

\theoremstyle{definition}
\newtheorem{defn}[thm]{Definition}

\theoremstyle{remark}

\makeatletter
\let\c@equation\c@thm
\makeatother
\numberwithin{equation}{section}

\title[]{Examples of open manifolds with positive Ricci curvature and non-proper Busemann functions}
\author[]{Jiayin Pan}
\address[Jiayin Pan]{Department of Mathematics, University of California,  Santa Cruz, CA, USA.}
\email{jpan53@ucsc.edu}
\thanks{J. Pan was supported by Fields Postdoctoral Fellowship during the preparation of this paper.}

\author[]{Guofang Wei}
\address[Guofang Wei]{Department of Mathematics, University of California,  Santa Barbara, CA, USA.}
\email{wei@math.ucsb.edu}
\thanks{G. Wei is partially supported by NSF DMS 2104704.}

\begin{document}
	
	\begin{abstract}
		We give the first example of an open manifold with positive Ricci curvature and a non-proper Busemann function at a point. This provides counterexamples to the longtime well-known open question of whether the Busemann function at a point of an open manifold with nonnegative Ricci curvature is proper. 
	\end{abstract}
	
	\maketitle
	
In the study of open (complete and noncompact) Riemannian manifolds, the Busemann function plays an essential role and has many applications. Recall that the Busemann function associated with a unit speed ray $c$ is defined as
$$b_c(x)=\lim\limits_{t\to+\infty} t-d(x,c(t)).$$ It measures a (renormalized) distance to a point at infinity along $c$. The Busemann function at a point $p\in M$ is $b_p(x)=\sup_{c} b_c(x)$, where $\sup$ is taken over all rays starting at $p$. See the work by Cheeger-Gromoll, Wu, Eschenburg-Heintze, Shiohama, Shen, and Sormani in \cite{CGsplit,CGsoul,Wu79,EH84,Shiohama1984,Shen93,Sor98} and the references in there for basic properties and applications of Busemann functions. It was also discussed by Yau in \cite[Section 2.2.1]{Yau}. %In this paper, we use the term Busemann function to refer to the one at some point unless otherwise noted.

When $M$ has nonnegative sectional curvature, Cheeger-Gromoll proved that the Busemann function $b_p$ is convex and proper \cite{CGsoul}, which means that the sublevel set
$$b_p^{-1}(-\infty, a]=\bigcap_c b_c^{-1}(-\infty,a]$$
is compact for all $a\in\mathbb{R}$, where the intersection is taken over all rays starting at $p$. When $M$ has nonnegative Ricci curvature, Cheeger-Gromoll proved that the Busemann function is subharmonic \cite{CGsplit}. Some literature, for example, \cite{Wu79,Shen93} used a different version of the Busemann function at a point
$$\hat{b}_p=\lim\limits_{t\to\infty} t-d(p,S_t(p)),$$
where $S_t(p)$ is the geodesic sphere of radius $t$ at $p$. We call this function $\hat{b}_p$ the \textit{spherical Busemann function} at $p$ in this paper. It follows from the definition that $b_p\le \hat{b}_p$ (see \cite[Section 2.1]{Shen93}), so a proper $b_p$ implies that $\hat{b}_p$ is also proper. The properness of $\hat{b}_p$ does not depend on the point $p\in M$ \cite[Lemma 3]{Shen93}. It has been an open question since the 70s whether the Busemann function $b_p$ is always proper for nonnegative Ricci curvature \cite{Cheeger_com}. 

In this paper, we give a negative answer to the question. In fact, we give open manifolds of positive Ricci curvature whose spherical Busemann function $\hat{b}_p$ is not proper. Consequently, its Busemann function $b_p$ is also not proper.

\begin{thm1} 
	Given any integer $n\ge 4$, there is an open $n$-manifold $N$ with positive Ricci curvature and a non-proper spherical Busemann function.
\end{thm1}

With nonnegative Ricci curvature, the Busemann function $b_p$ is proper when $M$ has linear volume growth by Sormani \cite{Sor98} or Euclidean volume growth by Shen \cite{Shen96}. Recall that we say an asymptotic cone $(X,x)$ of an open manifold with $\mathrm{Ric}\ge 0$ is polar, if for any $z\in X-\{x\}$, there is a ray from $x$ through $z$ (see Definition \ref{def_polar}). For manifolds with linear or Euclidean volume growth, their asymptotic cones are always polar as shown in the work of Sormani \cite{Sor98} and Cheeger-Colding \cite{ChCo96}, respectively. One might ask whether polar asymptotic cones would imply the properness of Busemann function. Our examples also show that this is not the case in general. 

\begin{thm2}
	There is an open manifold $N$ of dimension $16$ with positive Ricci curvature, a unique polar asymptotic cone, and a non-proper spherical Busemann function.
\end{thm2}

Menguy's example \cite{Men00} shows that a proper Busemann function does not imply polar asymptotic cones either.

Our examples are closely related to the construction by Nabonnand in 1980 \cite{Nab80}. See also related constructions by B\'{e}rard-Bergery \cite{Ber86} and the second named author \cite{Wei88}. Nabonnand constructed open manifolds as doubly warped products $[0,\infty)\times_f S^2 \times_h S^1$ with positive Ricci curvature. We use the Riemannian universal cover $N$ of a doubly warped product $M=[0,\infty) \times_f S^{k-1} \times_h S^1$ with suitable functions $f$ and $h$ as our examples in Theorems A and B. The investigation of these universal covers also leads to the first examples of Ricci limit spaces with non-integer Hausdorff dimension presented by the authors in \cite{PW}.

For Theorem A, we require that the function $h$ decreases to $0$ as $r\to\infty$. With a base point $p\in M$ at $r=0$ and $\tilde{p}\in N$ as a lift of $p$ in the universal cover, such a function $h$ and the classical Clairaut's theorem imply that any ray in $N$ starting at $\tilde{p}$ must project to a radial ray in $M$ starting at $p$ via the covering map $\pi:N\to M$ (Lemma \ref{ray_proj}). Then we can further show that the preimage $\pi^{-1}(p)$, which is unbounded, is contained in the level set of Busemann function $b_{\tilde{p}}$ with value $0$; this shows that $b_{\tilde{p}}$ is not proper (Theorem \ref{orbit_to_ray}). To prove that the spherical Busemann function is also not proper, in Section 1.3 we further study the distance from any point in $\pi^{-1}(p)$ to a geodesic sphere $S_t(\tilde{p})$. For Theorem B, we use a function $h$ with logarithm decay and show that any asymptotic cone of $N$ is isometric to the standard half-plane $[0,\infty)\times \mathbb{R}$. 

Acknowledgments: We would like to thank Christina Sormani for bringing up this question to us and for very helpful conversations. We would like to thank Jeff Cheeger, Zhongmin Shen, and Zhifei Zhu for helpful communication. We would like to thank the referee for their valuable suggestions that improves the paper.

\section{The examples and their geometry}
\subsection{Riemannian metrics on $M$}\label{subsec_metric}
We follow Nabonnand's construction in \cite{Nab80}. We include the details for readers' convenience.
On $[0,\infty)\times S^{k-1}\times S^1$, we define a doubly warped product metric
$$g=dr^2+f(r)^2ds_{k-1}^2+h(r)^2 ds_1^2,$$
where $ds_m^2$ is the standard metric on the unit sphere $S^{m}$. When $f$ and $h$ satisfy
$$f^{(\text{even})}(0)=0,\quad f'(0)=1,\quad h(0)>0,\quad f^{(\text{odd})}(0)=0,$$
$(M,g)$ is a Riemannian manifold diffeomorphic to $\mathbb{R}^k\times S^1$ (see \cite[Section 1.4.5]{Petersen}). $M$ with such a doubly warped product metric is denoted by $[0,\infty)\times_f S^{k-1} \times_h S^1$.

Let $H=\partial/\partial r$, $U$ a unit vector tangent to $S^{k-1}$, and $V$ a unit vector tangent to $S^1$. Then the metric has Ricci curvature
\begin{align*}
	&\mathrm{Ric}(H,H)=-\dfrac{h''}{h}-(k-1)\dfrac{f''}{f},\\
	&\mathrm{Ric}(U,U)=-\dfrac{f''}{f}+\dfrac{k-2}{f^2} \left[1-(f')^2\right]-\dfrac{f'h'}{fh},\\
	&\mathrm{Ric}(V,V)=-\dfrac{h''}{h}-(k-1)\dfrac{f'h'}{fh}.
\end{align*}

As indicated in the introduction, to construct our examples for Theorems A and B, we will use the Riemannian universal cover $N$ of $(M,g)$. We will choose different functions, $f$ and $h$, for the different theorems.

For Theorem A, we use the construction in Nabonnand's work \cite{Nab80}. We include some details below for readers' convenience.

\begin{prop}\label{calc}
	Let $$\varphi (x) = \dfrac{\sqrt{3}}{\pi}\int_0^x \dfrac{\arctan (u^3)}{u^2} du.$$
	Let $f$ be the solution of the ODE $f' = (1-\varphi (f))^{1/2}$ with $f(0) =0$ and let $h=f'$. Then the doubly warped product $[0,\infty)\times_f S^{k-1} \times_h S^1$ has positive Ricci curvature when $k\ge 3$. Moreover, $h$ strictly decreases to $0$ on $[0,\infty)$.
\end{prop}

\begin{proof}
	As 
	$$\int_0^\infty \dfrac{\arctan (u^3)}{u^2} du =  \dfrac{\pi}{\sqrt{3}}$$ 
	and $\varphi$ is strictly increasing on $[0,\infty)$,  we have 
	$$\varphi(0)=0,  \quad 0\le \varphi < 1, \quad \lim\limits_{x\to+\infty}\varphi(x)=1,  \quad \varphi'(0)=0.$$
	Since $f$ satisfies the ODE $f' = (1-\varphi (f))^{1/2}$ and $f(0)=0$, we have
	$$f'(0)=1,\quad 0<f'<1 \text{ on } (0,\infty),\quad \lim\limits_{r\to+\infty} f(r)=+\infty,\quad \lim\limits_{r\to+\infty} f'(r)=0,$$
	$$f''=\dfrac{1}{2}\left[ 1-\varphi(f) \right]^{-1/2}\cdot (-\varphi'(f) \cdot f')= -\dfrac{1}{2} \varphi'(f).$$
	Thus $f''<0$ on $(0,\infty)$. Also, an inductive argument shows that $f^{(\text{even})}(0)=0$. Together with our choice of $h=f'$, we see that $h$ strictly decreases to $0$ on $[0,\infty)$ and $[0,\infty)\times_f S^{k-1} \times_h S^1$ defines a smooth Riemannian metric.
	
	Next, we calculate the Ricci curvature. Observe that we have proved 
	$$ 0<f'<1,\quad f''<0,\quad h'<0$$
	Thus $\mathrm{Ric}(U,U) >0$ given $k\ge 2$. $h=f'$ implies $\mathrm{Ric}(V,V)= \mathrm{Ric}(H,H)$. We calculate
	$$\dfrac{h''}{h}=\dfrac{f'''}{f'}=-\dfrac{1}{2}\cdot \dfrac{\varphi''(f)\cdot f'}{f'}=-\dfrac{1}{2}\varphi''(f).$$
	Then by direct calculation,
	\begin{align*}
		\mathrm{Ric}(H,H) =& - \frac{f'''}{f'} -(k-1) \frac{f''}{f} \\
		= & \dfrac{1}{2}\varphi''(f) + (k-1)\cdot \dfrac{1}{2}\cdot \dfrac{\varphi'(f)}{f}\\
		= & \dfrac{\sqrt{3}}{2\pi}\left[\dfrac{3}{f^6+1}+(k-3)  \frac{\arctan (f^3)}{f^3}\right],
	\end{align*}
	which is positive given $k \ge 3$.
\end{proof}

For Theorem B, we use the functions defined by the first named author in \cite[Example B.9]{Pan21}:
$$f(r)=\dfrac{\sqrt{\ln 2}\cdot r}{\ln^{1/2}(2+r^2)},\quad h(r) =\dfrac{1} {\ln(2+r^2)}.$$
One can verify that the above $f$ and $h$ imply positive Ricci curvature when $k\ge 15$.

\subsection{Rays in the Riemannian universal cover $N$}\label{subsec_rays}

Let $N$ be the Riemannian universal cover of $M$. To study the rays in $N$, we first study the geodesics in $M$.

Given the metric $[0,\infty)\times_f S^{k-1}\times_h S^1$, we write each point in $M$ as $(r,x,v)$, where $r\ge 0$, $x\in S^{k-1}$, and $v\in S^1$. Note that $(0,x_1,v)$ and $(0,x_2,v)$ represent the same point in $M$.

We always fix a point $p=(0,x,0)$ at $r=0$ as our base point in $M$. Given any point $x\in S^{k-1}$, we denote a submanifold
$$C(x)=\{(r,\pm x,v)|r\ge 0, v\in S^1\}.$$
Note that $C(x)$ is totally geodesic and geodesically complete. The induced Riemannian metric on $C(x)$ is isometric to the warped product $\mathbb{R}\times_{\bar{h}} S^1$, where $\bar{h}$ is the extension of $h$ as an even function.

We recall Clairaut's theorem in classical differential geometry, which describes the geodesics on any $2$-dimensional warped product (see e.g.\cite{doCarmo}).

\begin{thm}\label{Clairaut}(Clairaut)
	Let $C=\mathbb{R}\times_h S^1$ be a warped product, where $h:\mathbb{R}\to \mathbb{R}_+$ is a smooth function. Let $\sigma(t)=(r(t),v(t))$ be a geodesic of unit speed in $C$. Then $h(r(t))\cos \theta(t)$ is independent of $t$, where $\theta(t)$ is the angle between the curve $\sigma$ and the parallel circle at $\sigma(t)$.
\end{thm}

In our context, $h=\bar{h}$ is a smooth even function with
$$h(0)>0,\quad h'(0)=0,\quad \lim\limits_{r\to\infty} h(r)=0,\quad h'(r)<0 \text{ on } (0,\infty).$$
Let $\sigma:[0,\infty)\to C$ be a geodesic of unit speed in $C=\mathbb{R}\times_h S^1$ with $\sigma(0)=(0,v)$, where $v\in S^1$. If $\theta(0)=0$, then $\sigma$ is a closed geodesic as the circle $\{r=0\}$. If $\theta(0)=\pi/2$, then $\sigma(t)=(t,v)$ is a ray. If $\theta(0)\not=0$ and $\not=\pi/2$, then by Clairaut's theorem,
$$h(0)\cos\theta(0)=h(r(t))\cos \theta(t)$$
for all $t\in\mathbb{R}$. In particular, $h(r(t))$ cannot be arbitrarily small. Thus $\sigma$ is bounded given that $h$ decays to $0$ as $r\to\infty$. Also, note that $\sigma$ must turn back at some point; in other words, it cannot stay asymptotically close to a parallel as $t\to\infty$. In fact, if $\sigma$ stays asymptotically close to $\{r=0\}$, then 
$$h(0)>h(0)\cos\theta(0)=r(h(t))\cos \theta(t)\to h(0)$$
as $t\to\infty$, which cannot happen; if $\sigma$ stays asymptotically close to $\{r=r_0\}$ for some $r_0\not= 0$, then part of the parallel $\{r=r_0\}$ is the limit of a sequence of minimal geodesics, but itself is not a geodesic by the choice of $h$. Therefore, when $\theta(0)\not=0$ and $\not=\pi/2$, we deduce that $\sigma$ oscillates between $\{r\ge 0\}$ and $\{r\le 0\}$ indefinitely; in particular, $\sigma$ crosses $\{r=0\}$ transversely infinitely many times.

With above, we can describe the geodesics of $M$ starting from $p$ as below. 
\begin{lem}\label{geod_p}
	Let $M=[0,\infty) \times_f S^{k-1} \times_h S^1$ be a doubly warped product, where $h$ decreases to $0$ as $r\to\infty$. Then \\
	(1) every geodesic of $C(x)$ with respect to the induced Riemannian metric is a geodesic in $M$.\\
	(2) every geodesic of $M$ starting at $p$ is contained in a $C(x)$ for some $x\in S^{k-1}$. \\
	(3) every geodesic of $M$ starting at $p$ has one of the following forms:\\
	i. a radial ray $\sigma(t)=(r(t),\pm x,v)$ for some $v\in S^1$;\\
	ii. a closed geodesic as the circle $\{r=0\}$;\\
	iii. $\sigma(t)=(r(t),s(t)x,v(t))$, where $s:[0,\infty)\to \{1,-1\}$ only changes value when $r(t)=0$; in this case, $\sigma$ is bounded and crosses $\{r=0\}$ transversely infinitely many times.
\end{lem}

\begin{proof}
	The submanifold $C(x)$ is totally geodesic in $M$ for all $x\in S^{k-1}$, thus (1) follows. Next, we prove (2). We identity the tangent plane $T_pM$ with $\mathbb{R}^k\times\mathbb{R}$ of a linear metric $$dr^2+r^2ds^2_{k-1}+du^2.$$
	A geodesic in $C(x)$ starting at $p$ has an initial vector $(t,\pm x,v)\in T_pM$, where $t\in[0,\infty)$ and $v\in\mathbb{R}$. Note that these initial vectors 
	$$\cup_{x\in S^{k-1}}\{(t,\pm x,v)| t\in[0,\infty), v\in \mathbb{R} \}$$
	exhaust all vectors in $T_pM$. Together with the fact that every $C(x)$ is totally geodesic, we conclude that a geodesic in $M$ with initial point $p$ must be contained in some $C(x)$ for a short time. Because $C(x)$ is also geodesically complete, the geodesic is entirely contained in this $C(x)$.
	
	With (2), the last statement follows directly from the description of the geodesics in the warped product $\mathbb{R}\times_{\bar{h}} S^1$ starting at $(0,v)$.
\end{proof}

\begin{lem}\label{ray_proj}
	Let $M=[0,\infty) \times_f S^{k-1} \times_h S^1$ be a doubly warped product with $\mathrm{Ric}>0$, where $h$ decreases to $0$ as $r\to\infty$. Let $c$ be a ray in the universal cover $N$ of $M$ starting at $\tilde{p}$, a lift of $p\in \{r=0\}$. Then $c$ projects to a radial ray $(r(t),x,v)$ in $M$ for some $x\in S^{k-1}$ and $y\in S^1$.
\end{lem}

\begin{proof}
	Note that any ray in the universal cover $N$ at $\tilde{p}$ projects to a geodesic in $M$ starting at $p$. It suffices to show that any geodesic described in Lemma \ref{geod_p}(ii) and (iii) cannot lift to a ray in $N$. Let $\sigma$ be a geodesic in $M$ described in Lemma \ref{geod_p}(ii) or (iii), and let $c$ be the lift of $\sigma$ at $\tilde{p}\in N$. In these cases, $\sigma$ is bounded. Let $R>0$ such that $\sigma$ is contained in the closed metric ball $\overline{B_R}(p)$. Then its lift $c$ is contained in $\pi^{-1}(\overline{B_R}(p))$, where $\pi$ is the covering map. 
	
	We argue by contradiction and suppose that $c$ is a ray in $N$. We write $\Gamma$ as the fundamental group $\pi_1(M,p)$, which is isomorphic to $\mathbb{Z}$. For each $j\in\mathbb{N}$, we can find $g_j\in \Gamma$ such that $g_j\cdot c(j)\in \overline{B_R}(\tilde{p})$. Then $g_j\cdot c|_{[0,2j]}$ is a minimal geodesic with length $2j$ and midpoint in $\overline{B_R}(\tilde{p})$. Passing to a subseqeuence if necessary, $g_j\cdot c|_{[0,2j]}$ converges to a line in $N$ as $j\to\infty$. By Cheeger-Gromoll's splitting theorem \cite{CGsplit}, $N$ splits off a line isometrically; a contradiction to positive Ricci curvature. 
\end{proof}

\subsection{Non-properness of the Busemann function in $N$}

We first prove that on the universal cover $N$, the Busemann function $b_{\tilde{p}}$ is not proper.

\begin{thm}\label{orbit_to_ray}
	Let $M=[0,\infty) \times_f S^{k-1} \times_h S^1$ be a doubly warped product with $\mathrm{Ric}>0$, where $h$ decreases to $0$ as $r\to\infty$. Let $p\in\{r=0\}$ and let $\gamma$ be a generator of $\pi_1(M,p)=\mathbb{Z}$. Then on the universal cover $(N,\tilde{p})$,
	$b_c(\gamma^l\tilde{p})=0$ for all $l\in\mathbb{Z}$ and all unit speed ray $c$ starting at $\tilde{p}$. Consequently, $b_{\tilde{p}}(\gamma^l\tilde{p})=0$ for all $l\in\mathbb{Z}$ and $b_{\tilde{p}}$ is nor proper.
\end{thm}

\begin{proof}
	By Lemma \ref{ray_proj}, $c$ projects to a radial ray $\bar{c}(t)=(t,x,v)$
	for some $x\in S^{k-1}$. Let $\alpha$ be a minimal geodesic from $\gamma^l\tilde{p}$ to $c(t)$, where $t>0$, and let $\bar{\alpha}$ be its projection to $M$. Among all curves from $p$ to $\bar{c}(t)$ in $C(x)$ that winds around $C(x)$ exactly $l$ rounds, $\bar{\alpha}$ has the shortest length. Let $\beta$ be a curve from $p$ to $\bar{c}(t)$ constructed as follows: $\beta$ moves along the radial ray $\bar{c}$ until the point $\bar{c}(t)$, then $\beta$ winds around the parallel circle at $\bar{c}(t)$ exactly $l$ rounds. We have
	$$d(\gamma^l\tilde{p},c(t))=\mathrm{length}(\bar{\alpha})\le \mathrm{length}(\beta)=t+l\cdot 2\pi h(t).$$
	It follows that
	$$b_c(\gamma^l\tilde{p})=\lim\limits_{t\to+\infty} t-d(\gamma^l\tilde{p},c(t))\ge \lim\limits_{t\to+\infty} -l\cdot 2\pi h(t)=0.$$
	On the other hand, because the covering map is distance non-increasing,
	$$d(\gamma^l\tilde{p},c(t))\ge d(p,\bar{c}(t))=t,$$
	thus 
	$$b_c(\gamma^l\tilde{p})=\lim\limits_{t\to+\infty} t-d(\gamma^l\tilde{p},c(t))\le 0.$$
	We conclude that $b_c(\gamma^l\tilde{p})=0$ and 
	$$b_{\tilde{p}}(\gamma^l\tilde{p})=\sup_c b_c(\gamma^l\tilde{p})=0,$$
	where $\sup$ is taken over all rays in $N$ starting at $\tilde{p}$. Because $\Gamma=\pi_1(M,p)$ is isomorphic to $\mathbb{Z}$, the orbit $\Gamma\tilde{p}$ is unbounded. It follows that the level set $b_{\tilde{p}}^{-1}(0)$ is non-compact.
\end{proof}

Next, we study the spherical Busemann function $\hat{b}_{\tilde{p}}$. We continue to use the notation $(C,g_0)=\mathbb{R}\times_h S^1$ as a $2$-dimensional warped product, where $h$ is an even function with $\lim_{r\to\infty} h(r)=0$. Let $\widetilde{C}$ be its Riemannian universal cover, whose metric is given by
$$\tilde{g}_0=dr^2+ h^2(r)dv^2 $$
defined on $\mathbb{R}^2$. Let $d_0$ be the induced distance function on $\widetilde{C}$. Below, we write any point in $\widetilde{C}$ as $(r,v)$. Note that the map $(r,v)\mapsto (r,v+l)$ is an isometry of $\widetilde{C}$ for each $l\in \mathbb{Z}$.

\begin{lem}\label{orbit_to_far}
	Fix an $l\in\mathbb{Z}$. Let $m_i=(r_i,v_i)\in \widetilde{C}$ be a sequence with $r_i\to+\infty$. Then
	$$\lim\limits_{i\to\infty} d_0((0,0),m_i)-d_0((0,l),m_i)\to 0.$$
\end{lem}

\begin{proof}
	We estimate that
	\begin{align*}
		& |d_0((0,0),(r_i,v_i))-d_0((0,l),(r_i,v_i))| \\
		=\ & |d_0((0,0),(r_i,v_i))-d_0((0,0),(r_i,v_i-l))|\\
		\le\ & d_0((r_i,v_i),(r_i,v_i-l))\\
		=\ & d_0((r_i,0),(r_i,l))\\
		\le\ & l\cdot 2\pi h(r_i) \to 0
	\end{align*}
	as $i\to\infty$.
\end{proof}

Let $M=[0,\infty) \times_f S^{k-1} \times_h S^1$ be a doubly warped product. The universal cover $N$ has its Riemannian metric as a doubly warped product $[0,\infty)\times_f S^{k-1} \times_h \mathbb{R}$:
$$dr^2+f^2(r) ds_{k-1}^2 + h^2(r) dv^2.$$
We use $(r,x,v)$ to denote points in $N$, where $r\in[0,\infty)$, $x\in S^{k-1}$, and $v\in\mathbb{R}$. Note that for each $l\in\mathbb{Z}$, the point $\gamma^l\tilde{p}$ corresponds to $(0,x,l)$ for any $x\in S^{k-1}$.

\begin{lem}\label{dist_equal}
	$d_N(\gamma^l\tilde{p},(r,x,v))=d_0((0,l),(r,v))$ holds for all $(r,x,v)\in N$.
\end{lem}

\begin{proof}
	We fix a point $x\in S^{k-1}$. Let $\sigma$ be a minimal geodesic from $\gamma^l \tilde{p}=(0,x,l)$ to $(r,x,v)$. Then $\pi\circ\sigma$ is a geodesic in $M$ from $p$. By Lemma \ref{geod_p}, $\pi\circ\sigma$ is entirely contained in $C(x)$ and thus naturally defines a geodesic $\sigma_0$ in $(C,g_0)$ with the same length. Let $\tilde{\sigma}_0$ be the lift of $\sigma_0$ to $\widetilde{C}$ at $(0,l)$. Note that $\tilde{\sigma}_0$ is also minimal; otherwise, we can correspondingly define a shorter curve in $N$ from $(0,x,l)$ to $(r,x,v)$. Then the result follows.
\end{proof}

\begin{thm}\label{s_Busemann_nonproper}
	Let $M=[0,\infty) \times_f S^{k-1} \times_h S^1$ be a doubly warped product with $\mathrm{Ric}>0$, where $h$ decreases to $0$ as $r\to\infty$. Let $(N,\tilde{p})$ be the Riemannian universal cover of $(M,p)$, where $p\in M$ is a point at $r=0$. Then the Busemann function $\hat{b}_{\tilde{p}}$ is not proper.
\end{thm}

\begin{proof}
	Similar to the result of Theorem \ref{orbit_to_ray}, we shall prove that $\hat{b}_{\tilde{p}}(\gamma^l\tilde{p})=0$ for all $l\in\mathbb{Z}$. Then the level set $\hat{b}_{\tilde{p}}^{-1}(0)$ is non-compact. We fix an $l\in\mathbb{Z}$. The desired lower bound follows from Theorem \ref{orbit_to_ray}:
	$$\hat{b}_{\tilde{p}}(\gamma^l\tilde{p})\ge b_{\tilde{p}}(\gamma^l\tilde{p})=0.$$ It remains to prove that $\hat{b}_{\tilde{p}}(\gamma^l\tilde{p})\le 0$. 
	
	Let $t_i\to+\infty$ be a sequence and let $z_i=(r_i,x_i,v_i)\in S_{t_i}(\tilde{p})$ such that 
	$$d_N(\gamma^l\tilde{p},S_{t_i}(\tilde{p}))=d_N(\gamma^l\tilde{p},z_i).$$
	
	%\textit{Case 1:} $r_i\to+\infty$ for a subsequence. After passing to this subsequence, by Lemma \ref{orbit_to_far}, we have
	%$$\hat{b}_{\tilde{p}}(\gamma^l \tilde{p})=\lim\limits_{i\to\infty} d_0((0,0),(r_i,v_i))-d_0((0,l),(r_i,v_i))=0.$$
	
	%\textit{Case 2:} $\{r_i\}$ is uniformly bounded. 
	
	In $(\widetilde{C},d_0)$, let $\sigma_i$ be a minimal geodesic from $(0,l)$ to $(r_i,v_i)$. Let $m_i=(R_i,w_i)$ be a point on $\sigma_i$ such that the $r$-component $R_i$ of $m_i$ is the maximal among all points on $\sigma_i$. Note that $R_i\to\infty$. Otherwise, by Lemmas \ref{geod_p} and \ref{dist_equal}, such $\sigma_i$ implies that $N$ contains a sequence of minimal geodesics with diverging length in $\pi^{-1}(\overline{B_R}(p))$ for some $R>0$, where $\pi:N \to M$ is the covering map. Then similar to Lemma \ref{ray_proj}, we can use suitable covering transformations to derive a limit line in $N$. This contradicts with Cheeger-Gromoll's splitting theorem and the positive Ricci curvature of $N$.
	
	By Lemma \ref{dist_equal}, triangle inequality, and the choice of $m_i$,
	\begin{align*}
		&t_i - d_N(\gamma^l \tilde{p},S_{t_i}(\tilde{p})) = d_0((0,0),(r_i,v_i))-d_0((0,l),(r_i,v_i))\\
		\le \ & d_0((0,0),m_i)+d_0(m_i,(r_i,v_i))-d_0((0,l),m_i)-d_0(m_i,(r_i,v_i))\\
		= \  & d_0((0,0),m_i)-d_0((0,l),m_i).
	\end{align*}
	Then by Lemma \ref{orbit_to_far},
	$$\hat{b}_{\tilde{p}}(\gamma^l \tilde{p})= \lim\limits_{i\to\infty}t_i - d_N(\gamma^l \tilde{p},S_{t_i}(\tilde{p}))\le 0.$$
	
	We complete the proof of $\hat{b}_{\tilde{p}}(\gamma^l \tilde{p})= 0$ for all $l\in\mathbb{Z}$.
\end{proof}

Theorem A follows from Proposition \ref{calc} and Theorem \ref{s_Busemann_nonproper}.

\subsection{Polar asymptotic cones}

We recall the following definition.

\begin{defn}\label{def_polar}
	Let $M$ be an open manifold with $\mathrm{Ric}\ge 0$. An asymptotic cone of $M$ (or a tangent cone of $M$ at infinity) is the limit of a pointed Gromov-Hausdorff convergent sequence
	$$(r_i^{-1}M,p)\overset{GH}\longrightarrow (X,x),$$
	where $r_i\to\infty$. We say that $(X,x)$ is polar, if for any $z\in X-\{x\}$, there exists a ray starting from $x$ and going through $z$.
\end{defn}

\begin{proof}[Proof of Theorem B]
	As indicated, we use $M$ as the doubly warped product with $h(r)=\ln^{-1}(2+r^2)$ in Section \ref{subsec_metric}, then $N$ as the Riemannian universal cover of $M$. By Theorem \ref{s_Busemann_nonproper}, $N$ has a non-proper spherical Busemann function because $h$ decreases to $0$ as $r\to +\infty$. It remains to show that any asymptotic cone $(Y,y)$ of $N$ is polar. In fact, we show that $Y$ is the standard halfplane $[0,\infty)\times\mathbb{R}$ and $y=(0,0)$. We will present two different proofs below. The first one directly calculates the limit metric. The second one follows from the arguments by the first named author in \cite[Appendix B]{Pan21} and by the authors in \cite{PW}. While the second one is indirect compared to the first one, it may provide independent interests to some readers.
	
	\textit{Proof I:}
	Let $\lambda>1$. Under a change of variables $s=\lambda^{-1} r$ and $w=(2\lambda \ln\lambda)^{-1} v$ on $N=[0,\infty) \times_f S^{k-1} \times_h \mathbb{R}$ with
	$$g_N=dr^2 + f^2(r) d_{k-1}^2+ h^2(r) dv^2,$$ we have
	\begin{align*}
		\lambda^{-2}g_N&=\lambda^{-2} \left[dr^2+ \dfrac{(\ln 2)\cdot r^2}{\ln(2+r^2)} ds_{k-1}^2 + \dfrac{1}{\ln^{2}(2+r^2)} dv^2 \right] \\
		&= ds^2 + \left[ \dfrac{(\ln 2)\cdot s^2}{\ln(2+\lambda^2 s^2)} \right] ds_{k-1}^2 +\left[\dfrac{2\ln\lambda}{\ln(2+\lambda^2s^2)}\right]^2 dw^2 \\
		&=: ds^2 + f_\lambda(s)^2 ds_{k-1}^2 + h_\lambda(s)^2 dw^2.
	\end{align*}
	Note that as $\lambda\to+\infty$, $f_\lambda$ and $h_\lambda$ converge uniformly to $0$ and $1$, respectively, on every compact subset of $(0,\infty)$. Thus the limit space has the standard metric $ds^2+dw^2$ on the open halfplane $(0,\infty)\times\mathbb{R}$. Hence its metric completion, the closed standard halfplane, is the asymptotic cone. 
	
	\textit{Proof II:} Let $\gamma$ be a generator of $\Gamma=\pi_1(M,p)$ and let $c_l$ be a minimal representing loop of $\gamma^l\in \Gamma$, where $l\in \mathbb{Z}$. In \cite[Example B.9]{Pan21}, we have estimated $\mathrm{length}(c_l)$ and $d_H(p,c_l)$, where $d_H$ means the Hausdorff distance, and showed that
	$$\lim\limits_{l\to\infty} \dfrac{d_H(p,c_l)}{\mathrm{length}(c_l)}=0.$$
	Let $\tilde{c_l}$ be the lift of $c_l$ starting at $\tilde{p}$; note that $\tilde{c_l}$ is a minimal geodesic from $\tilde{p}$ to $\gamma^l\cdot \tilde{p}$. For each $l\in\mathbb{Z}$, we put $\alpha_l=\gamma^{-l}\cdot \tilde{c}_{2l}$, which is a minimal geodesic from $\gamma^{-l}\tilde{p}$ to $\gamma^l\tilde{p}$. Then $\alpha_l$ is contained in the tubular neighborhood $B_{\epsilon_l}(\Gamma \tilde{p})$, where $\epsilon_l$ satisfies  
	$$\lim\limits_{l\to\infty} \dfrac{\epsilon_l}{d(\tilde{p},\gamma^{2l} \tilde{p})}=0.$$
	Now following \cite[Section 1.3]{PW}, we study the asymptotic cones of $N$. Let $r_i\to\infty$ be a sequence and we consider the equivariant Gromov-Hausdorff convergence
	\begin{center}
		$\begin{CD}
			(r^{-1}_iN,\tilde{p},\Gamma) @>GH>> 
			(Y,y,G)\\
			@VV\pi V @VV\pi V\\
			(r^{-1}_iM,p) @>GH>> (X,x).
		\end{CD}$
	\end{center}
	By \cite[Lemma 3.4]{FY92}, $X$ is isometric to the quotient space $Y/G$. Because $r^{-1}f(r)\to 0$ and $h(r)\to 0$ as $r\to\infty$ in our construction, $X$ is isometric to a half-line $[0,\infty)$. We view $Y$ as attaching orbits to $[0,\infty)$. By construction, each limit orbit is homeomorphic to $\mathbb{R}$ and $Y$ is homeomorphic to $[0,\infty)\times \mathbb{R}$, where the orbit $Gy$ corresponds to $\{0\}\times \mathbb{R}$. For any $g\in G-\{e\}$, there is a sequence $l_i\to\infty$ such that
	$$(r_i^{-1}N,\tilde{p},\gamma^{l_i})\overset{GH}\longrightarrow (Y,y,g).$$
	Because $\alpha_{l_i}$, a minimal geodesic from $\gamma^{-l_i}\tilde{p}$ to $\gamma^{l_i}\tilde{p}$, is contained in $B_{\epsilon_{l_i}}(\Gamma\tilde{p})$, where $r_i^{-1}\epsilon_{l_i}\to 0$ as $i\to\infty$, $\alpha_{l_i}$ converges to a limit minimal geodesic $\alpha_g$ from $g^{-1}y$ to $gy$ that is contained in the limit orbit $Gy$ and passes through $y$. Let $g_j\in G$ be a sequence with $d(g_jy,y)\to\infty$. Then $\alpha_{g_j}$ subconverges to a line in $Gy$ through $y$. By Cheeger-Colding's splitting theorem \cite{ChCo96}, $Y$ splits isometrically as $[0,\infty)\times \mathbb{R}$, which is polar at the base point $y=(0,0)$ (Cf. \cite[Theorem 1.3]{Pan21}). 
\end{proof}	

Comparing the asymptotic cone in Theorem B with the ones in \cite{PW}, we remark that in \cite{PW} the warping function $h$ has polynomial decay, which is different from the logarithm decay in Theorem B, then the corresponding asymptotic cone $Y$ does not split off any lines.
	
\bibliographystyle{plain} 
%\bibliography{bin}
%\end{document}	

\end{document}